\documentclass{amsart}
\usepackage{amssymb,amsfonts,color}
\usepackage{enumerate}

\ifx\pdfpageheight\undefined
   \usepackage[dvips,colorlinks=true,linkcolor=blue,citecolor=red,%
      urlcolor=green]{hyperref}
   \usepackage[dvips]{graphicx}
   \makeatletter
   \edef\Gin@extensions{\Gin@extensions,.mps}
   \makeatother
\else
 \usepackage[pdftex]{graphicx}
   \usepackage[bookmarksopen=false,pdftex=true,breaklinks=true,%
      backref=page,pagebackref=true,plainpages=false,%
      hyperindex=true,pdfstartview=FitH,colorlinks=true,%
      pdfpagelabels=true,colorlinks=true,linkcolor=blue,%
      citecolor=red,urlcolor=green,hypertexnames=false%
      ]%
   {hyperref}
\fi
\usepackage{bm}

\usepackage{amsmath,amssymb,amsfonts}
\newtheorem{theorem}{Theorem}
\newtheorem{lemma}{Lemma}

\newtheorem{proposition}{Proposition}

\theoremstyle{definition}
\newtheorem{definition}{Definition}

\newtheorem{notation}{Notation}

\theoremstyle{remark}
\newtheorem{remark}{Remark}

\definecolor{DarkBlue}{rgb}{0,0.1,0.55}

\numberwithin{equation}{section}

\newcommand {\hide}[1]{}

\newcommand {\junk}[1]{}

\newcommand {\R} {\mathrm{R}}

\newcommand {\Q}         {\mathbb{Q}}

\newcommand {\eps} {{\varepsilon}}

\newcommand{\card}{\mathrm{card}}

\def\addots{\mathinner{\mkern1mu
\raise1pt\vbox{\kern7pt\hbox{.}}
\mkern2mu\raise4pt\hbox{.}\mkern2mu
\raise7pt\hbox{.}\mkern1mu}}

\newcommand{\HH}  {\mbox{\rm H}}

\renewcommand{\Im}{\mathrm{Im}}

\newcommand{\clos}{\mathrm{clos}}

\begin{document} 
\title{
An o-minimal Szemer\'edi-Trotter theorem}

\begin{abstract}
We  prove an analog of the Szemer\'edi-Trotter theorem in the plane for definable curves and points in any o-minimal structure
over an arbitrary real closed field $\R$. One new ingredient in the proof is
an extension of the  well known crossing number inequality for graphs to the case of embeddings 
in any o-minimal structure over an arbitrary real closed field.
\end{abstract}
\author{Saugata Basu}
\address{Department of Mathematics,
Purdue University, West Lafayette, IN 47906, U.S.A.}
\email{sbasu@math.purdue.edu}

\author{Orit E. Raz}
\address{
School of Mathematics, 
Institute for Advanced Study, Princeton NJ 08540, USA}
\email{oritraz@ias.edu}

\thanks{Part of this research was performed while the authors were visiting the Institute for Pure and Applied Mathematics (IPAM), which is supported by the National Science Foundation.
The first author was partially supported by NSF grants
CCF-1618981 and DMS-1620271. The second author was partially supported by NSF grant DMS-1128155.}

\maketitle

\section{Introduction}
The Szemer\'edi-Trotter theorem \cite{Szemeredi-Trotter} on incidences between lines and points 
in $\mathbb{R}^2$ is one of the first non-trivial results in quantitative incidence theory and is considered a foundational result in discrete geometry and 
extremal combinatorics. The statement of the theorem is as follows.

\begin{theorem}\cite{Szemeredi-Trotter}
\label{thm:sz-t}
There exists a constant $C  >0$, such that
given any finite set $\Gamma$ of lines in $\mathbb{R}^2$, and a finite set $\Pi$ of points in $\mathbb{R}^2$,
\[
\card(\{(p,\gamma) | p \in \Pi,  \gamma \in \Gamma, p \in \gamma\}) \leq C \cdot (\card(\Gamma)^{2/3}\cdot \card(\Pi)^{2/3} + \card(\Gamma)+ \card(\Pi)).
\]
\end{theorem}

Theorem \ref{thm:sz-t} has been generalized later in many different ways -- to algebraic curves instead of lines \cite{Pach-Sharir,SSZ2015}, incidences between points and algebraic hypersurfaces in higher dimensions \cite{Zahl2013, Sharir-Solomon,Basu-Sombra2015}, replacing $\mathbb{R}$ by $\mathbb{C}$ \cite{Toth,Zahl2015} etc.

It was shown in \cite{Basu9}, that many quantitative results in the theory of arrangements of  semi-algebraic sets of  ``constant description complexity'' could be generalized to the setting where the 
objects in the arrangements are not necessarily semi-algebraic sets of constant description complexity,
but rather are restricted to be the fibers of some fixed definable map in any o-minimal structure over a real closed field $\R$. 
(We refer the reader to \cite{Dries,Michel2} for the definition and basic results on
o-minimal structures.)
More recently, o-minimal generalizations of results in combinatorial geometry have become a very active topic of research \cite{Starchenko-Chernikov} (see also the survey article \cite[\S 6]{Scanlon}). 

Recently, Fox et al. \cite[Theorem 1.1]{Pach-et-al} obtained a very far reaching generalization of Theorem \ref{thm:sz-t}, by extending it to the case of  semi-algebraic curves
of fixed description complexity. It is thus a natural question if incidence results, such as the Szemer\'edi-Trotter theorem, and its various generalizations can also be extended to the more general setting of o-minimal geometry.
In this paper we prove the following o-minimal generalization of the Szemer\'edi-Trotter theorem in the plane.

\subsection{Main Result}
For the rest of the paper we fix an o-minimal structure over a real closed field $\R$. 

We prove the following theorem.
\begin{theorem}
\label{thm:main'}
Let $V$ be a definable subset of $E_1\times E_2$, where, for $i=1,2$,  
$E_i$ is a definable set  of dimension at most two.
Then one of the following holds.
\begin{enumerate}[(i)]
\item
There exists a constant $C > 0$, which depends on $V, E_1, E_2$, such that
for every finite subsets $P\subset E_1$, $Q\subset E_2$,
$$
\card(V\cap (P\times Q)) \leq C \cdot (\card(P)^{2/3}\cdot\card(Q)^{2/3}+\card(P)+\card(Q)).
$$
\item
There exist  definable subsets $\alpha\subset E_1$ and $\beta\subset E_2$, with $\dim(\alpha),\dim(\beta) \geq 1$, such that
$$
\alpha\times \beta\subset V.
$$
\end{enumerate}
\end{theorem}

Theorem \ref{thm:main'} will be an immediate consequence of the following simpler version.

Let $E = \R^2$ and
$F = E \times E$, and $\pi_1,\pi_2:F \rightarrow E$, and $\sigma_1,\sigma_2:E \rightarrow \R$, be the canonical
projections. For definable subsets $P,Q \subset E$, there is a (definable) canonical injection $P \times Q \rightarrow F$,
and we will slightly abuse notation and consider $P \times Q$ to be a (definable) subset of $F$. 

\begin{theorem}\label{thm:main}
Let $V\subset F$ be a definable subset of $F$. 
Then one of the following holds.
\begin{enumerate}[(i)]
\item
\label{itemlabel:main:i}
There exists a constant $C_V > 0$, such that
for every finite subsets $P,Q\subset E$,
$$
\card(V\cap (P\times Q)) \leq C_V \cdot (\card(P)^{2/3}\cdot\card(Q)^{2/3}+\card(P)+\card(Q)).
$$
\item
\label{itemlabel:main:ii}
There exist  definable subsets $\alpha,\beta\subset E$, $\dim(\alpha),\dim(\beta) \geq 1$, such that
$$
\alpha\times \beta\subset V.
$$
\end{enumerate}
\end{theorem}

\begin{remark}
Simultaneously with our paper Chernikov, Galvin and Starchenko \cite{2016arXiv161200908C} also announced a similar result. Their result is more general than ours 
(it applies to general distal structures), but the techniques behind their proof are quite different.
\end{remark}

\begin{remark} 
Note that in case property 
\eqref{itemlabel:main:ii}  holds, then for any finite $P\subset \alpha$, $Q\subset \beta$
we have $\card(V\cap(P\times Q))=\card(P )\cdot \card(Q)$.
\end{remark}

\begin{remark}
Note also that in case the fibers of $\pi_1|_V,\pi_2|_V$ are of dimension at most one, each point in $(p,q) \in V \cap (P \times Q)$ corresponds to an ``incidence" of the point $p$ with the definable curve
$\pi_1(\pi_2^{-1}(q) \cap V)$. This is the sense in which Theorem \ref{thm:main} can be thought of as an o-minimal version of  Theorem \ref{thm:sz-t}.
Also notice that the statement of Theorem \ref{thm:main} is symmetric with respect to the sets $P,Q$.
\end{remark}

\begin{remark}
The formulation of Theorem \ref{thm:main} is inspired by the main results 
and the proofs in
\cite{MPVD, Raz-Sharir-deZeeuw1,Raz-Sharir-deZeeuw2}.
In these papers one is interested in a bound on the 
cardinality of a set of the form $V\cap (P\times Q)$,
where $V\subset \mathbb R^4$ is an algebraic variety of fixed degree, and each of $P,Q$
is a finite subset of $\mathbb R^2$ (of arbitrarily large cardinality).
This problem is then interpreted as an incidence problem between points and curves in $\mathbb{R}^2$.
As in Theorem~\ref{thm:main}, 
there exists exceptional varieties for which only a trivial upper bound on the number of incidences can hold, 
and these exceptional cases are identified (see Lemma~\ref{lem:Usmall} below).
\end{remark}

\begin{remark}(Tightness).
Since the  Szemer\'edi-Trotter theorem is known to be tight, up to a constant, the bound in property \eqref{itemlabel:main:i} in Theorem \ref{thm:main} cannot be improved.  As noted previously,
Chernikov, Galvin and Starchenko  \cite{2016arXiv161200908C} have  also considered the problem of proving incidences between points and definable curves
where the definable curves come from a definable family of some fixed $d \geq 2$, and have proved a generalization of 
Theorem \ref{thm:main} using the technique of ``cuttings''.  It is likely that the method used in this paper can also be extended to this general situtation.   However, the bound in \cite{2016arXiv161200908C} 
for definable families of dimension $d > 2$ is not believed to be sharp (see for example \cite[Theorem 1.3]{Sharir-Zahl}), 
and we do not attempt to prove this more general result in this paper.
\end{remark}

We first derive Theorem \ref{thm:main'} from Theorem \ref{thm:main}.

\begin{proof}[Proof of Theorem \ref{thm:main'}]

For each $i=1,2$, suppose that $E_i\subset \R^{d_i}$
and consider a cylindrical definable cell decomposition\footnote{For details about {\it cylindrical definable cell decomposition} see e.g \cite[\S 2]{Michel2} and references therein.} of $\R^{d_i}$ adapted to $E_i$.
Note that each cell of the decomposition is at most two-dimensional, by our assumption
on $E_1,E_2$.

For each cell $C\subset E_1$ (resp., $D\subset E_2$) of the decomposition
there exists a definable homeomorphism $\theta_{C}:C\to \R^{\dim(C)}$ (resp., 
$\varphi_{D}:D\to \R^{\dim(D)}$). 
Let 
 $f_{C,D} = \theta_C \times \varphi_D :C\times D\to \R^{\dim(C)}\times \R^{\dim(D)}$ be the homeomorphism induced by $\theta_C$ and 
 $\varphi_D$.
Applying Theorem~\ref{thm:main} to 
the sets
$$
\widetilde V:=f_{C,D}(V\cap (C\times D))\subset \R^{\dim(C)}\times \R^{\dim(D)},
$$
$\widetilde P:=\theta_{C}(C\cap P)$, and $\widetilde Q:=\varphi_{D}(D\cap Q)$,
and using the invertability of $f_{C,D}$,
we conclude that either 
there exist definable  $\alpha \subset E_1$, $\beta\subset E_2$ of dimension at least one, such that $\alpha\times\beta\subset V$,
in which case we are done, or 
$$
\card\left(V\cap(C\times D)\cap (P\times Q)\right)\le C_{\widetilde V} \cdot (\card(P)^{2/3}\cdot\card(Q)^{2/3}+\card(P)+\card(Q)).
$$

Repeating the argument for each pair of cells $C,D$, and recalling that the number of cells is finite, 
this completes the proof 
of the theorem.
\end{proof}

The rest of the paper is devoted to the proof of Theorem \ref{thm:main}.

\subsection{Outline of the proof of Theorem \ref{thm:main}}
There are two main approaches to recent proofs of the classical Szemer\'edi-Trotter theorem in the plane.
The first approach uses the well known technique of efficient partitioning the plane (using either the notion of cutting \cite{Matousek} or 
the newer method of polynomial partitioning \cite{Guth-Katz}) adapted to the given set of points, and then using a divide-and-conquer argument. The technique of polynomial partitioning is as yet not available over o-minimal structures, and the ``cutting lemma'' argument while feasible to generalize to o-minimal structures is technically complicated
(this is the approach taken in \cite{2016arXiv161200908C}).
The second method, that we adapt in this paper,  
is due to Sz\'ekely \cite{Szekely} who used an argument based on the 
``crossing number inequality'' for abstract graphs
due to Ajtai et al \cite{ACNS} and independently Leighton \cite{Leighton}.
The definition of the ``crossing number'' 
of a graph needs to be adapted to the o-minimal setting.

We define the \emph{definable crossing number} of graphs (Definition \ref{def:crossing}), 
in terms of \emph{definable embeddings}  of graphs in $\R^2$
(Definition \ref{def:embedding}),
and prove the definable analog of the Euler relation for definable embeddings of planar graphs (Lemma \ref{lem:euler}). 
The proofs of some of the lemmas use the existence of a good (co)-homology theory for general 
o-minimal structures \cite{Woerheide}, and in particular we use the o-minimal 
version of Alexander-Lefschetz duality theorem in $\R^2$ by Edmundo and Woerheide~\cite{Edmundo-Woerheide2009}.

The analog of the crossing number inequality in this definable setting 
(Lemma \ref{lem:cross}) then follows 
from Euler's relation, using a now-standard probabilistic argument
introduced first in \cite[p. 285]{Alon-Spencer}.

Using the crossing number inequality we then prove (see Lemma \ref{lem:lines}), following Sz\'ekely's argument~\cite{Szekely},
that given a finite set of points and a finite set of definable curves belonging to a fixed definable family,
satisfying a certain combinatorial condition on incidences (namely, that their incidence graph does not contain
a $K_{2,k}$ or a $K_{k,2}$ 
(where $K_{s,t}$ denotes the complete bi-partite graph with the two vertex sets of cardinalities $s$ and $t$ respectively)
 for some fixed $k$), also satisfies the Szemer\'edi-Trotter bound (with the constant depending on $k$).

Finally, 
in Section \ref{sec:proof-of-main} we establish the
dichotomy in Theorem \ref{thm:main} via Lemmas 
\ref{lem:lines},
\ref{lem:prod},
and \ref{lem:Usmall}. 
This proof of the key Lemma~\ref{lem:Usmall} uses ideas introduced in \cite{Raz-Sharir-deZeeuw1} for treating the algebraic case.
Theorem \ref{thm:main} then follows immediately from these lemmas.

\section{Incidence bound for definable pseudo-lines}
\label{sec:pseudolines}
We have the following bound for incidences between points and definable pseudo-lines in $\R^2$.
\begin{lemma}\label{lem:lines}
Let $V \subset F$ be as in Theorem \ref{thm:main}, and $k > 0$. 
There exists a constant $C_{V,k}$, depending only on $V$ and the parameter $k$, with the following property.
For every set $\Pi$ of $m$ points in $\R^2$, and a set
$\Gamma$  of $n$ definable curves\footnote{A {\it definable curve} 
is a one-dimensional definable set.} in $\R^2$, where each 
$\gamma \in \Gamma$  is of the form $\pi_1(\pi_2^{-1}(q)\cap V)$ for some $q \in \R^2$,
such that:
\begin{enumerate}[(a)]
\item \label{itemlabel:lines:a}
Every pair of distinct $p,p'\in \Pi$ belongs to at most $k$ curves of $\Gamma$, and
\item \label{itemlabel:lines:b}
Every pair of distinct  $\gamma,\gamma'\in \Gamma$ intersect in at most $k$ points,
\end{enumerate}
we have 
$$
I(\Pi,\Gamma) :=\card\left(\{(p,\gamma)\mid p \in \Pi, \gamma \in \Gamma, p \in \gamma\}\right) \leq C_{V,k} \cdot\left(m^{2/3}n^{2/3}+m+n\right).
$$
\end{lemma}

The rest of this section is devoted to proving Lemma \ref{lem:lines}. 
In order to prove Lemma \ref{lem:lines} we need to use the crossing number inequality (see Lemma \ref{lem:cross} below).
The proof of this inequality for topological embeddings of graphs in $\mathbb{R}^2$ is quite well-known (see Alon and Spencer~\cite[p. 285]{Alon-Spencer}) and applies in the definable context as well (for $\R = \mathbb{R}$).
Hence,  if one is only interested in the case $\R = \mathbb{R}$, one can skip the much of the rest of this section and proceed directly to
the proof of Lemma \ref{lem:lines}.

We begin with a basic result that we will use from algebraic topology over arbitrary o-minimal structures.

\subsection{Preliminaries  from o-minimal algebraic topology}
Singular homology and cohomology groups 
for definable sets of arbitrary o-minimal structures
have been defined  by Woerheide \cite{Woerheide}. This homology theory obeys the standard 
axioms of Eilenberg and Steenrod. In particular, there exist exact sequences for pairs and so on.  
 We will use the following result which is an immediate consequence of 
 Alexander-Lefschetz duality for definable manifolds obtained by Edmundo and Woerheide~\cite{Edmundo-Woerheide2009}.

\begin{proposition}
\label{prop:alexander}
Let $A$ be a closed and bounded definable subset of $\R^2$. Then, the number of definably connected components of
$\R^2 - A$ equals $ \dim_\Q \HH^1(A,\Q) + 1$.
\end{proposition}

\begin{proof}
It follows from the Alexander-Lefschetz duality theorem for definable manifolds \cite[Theorem 3.5]{Edmundo-Woerheide2009} that there is an isomorphism,
\begin{equation}
\label{eqn:prop:alexander:1}
\HH^1(A,\Q) \cong \HH_1(\R^2,\R^2- A,\Q).
\end{equation}

It follows now from \eqref{eqn:prop:alexander:1} and the homology exact sequence of the pair $(\R^2,\R^2 - A)$  (see for example \cite[\S 5]{Edmundo-Woerheide2008}), namely,
\[
\cdots  0 \cong \HH_1(\R^2,\Q ) \rightarrow \HH_1(\R^2,\R^2 - A,\Q) \rightarrow 
\HH_0(\R^2 - A,\Q) \rightarrow  \HH_0(\R^2,\Q) \cong \Q \cdots,
\]
that 
\begin{eqnarray*}
\dim_\Q \HH_0(\R^2 - A,\Q) &=& \dim_\Q \HH_1(\R^2,\R^2 - A,\Q)  +   \dim_\Q  \HH_0(\R^2,\Q) 
\end{eqnarray*}

It now follows from \eqref{eqn:prop:alexander:1}, and the fact that
\[
\dim_\Q  \HH_0(\R^2,\Q)  = 1,
\]
that
\[
\dim_\Q \HH_0(\R^2 - A,\Q) = \dim_\Q \HH^1(A,\Q) + 1.
\]
Finally, observe that 
that for any definable set $X$, $\dim_\Q \HH_0(X,\Q)$ equals the number of definable connected components of $X$. For $X$ closed and bounded, this is a consequence
of the fact $X$ admits a finite definable triangulation, 
the isomorphism between o-minimal simplicial and singular homology \cite[Theorem 1.1]{Edmundo-Woerheide2008},
and the corresponding fact for simplicial homology theory.  An arbitrary definable set is definably homotopy equivalent to
a closed and bounded one, and the fact that  $\dim_\Q \HH_0(X,\Q)$ equals the number of definable connected components of $X$, follows from the above and the homotopy 
invariance property of o-minimal singular homology \cite[\S 5]{Edmundo-Woerheide2008}. 
The proof is now complete once we note that
$\dim_\Q \HH_0(\R^2 -A,\Q)$ equals the number of definably connected components of $\R^2 -A$.
\end{proof}

We next introduce some standard notation and definitions from graph theory but adapted to the o-minimal context.

\subsection{Graph-theoretic notation and definition}
\begin{notation}
A {\it simple graph} $G$ is a pair $(V(G),E(G))$ where $V(G)$ is a finite nonempty set and 
$E(G)$ is a set of subsets of $V(G)$, each of cardinality 2.
We refer to the elements of $V(G)$ and $E(G)$ as {\it vertices} and {\it edges}, respectively. 
A {\it path} in $G$ is a sequence $(w_1,\ldots,w_r)$ of elements of $V(G)$, such that $\{w_i,w_{i+1}\}\in E(G)$, for every $1\le i\le r-1$.
A {\it cycle} in $G$ is a path $(w_1,\ldots,w_r)$ such that $w_1=w_r$.
We say that $G$ is {\it connected}, if for every $u\neq v\in V(G)$ there exists a
path $(w_1,\ldots,w_r)$ in $G$ such that $w_1=u$ and $w_r=v$.
\end{notation}

\begin{definition}
\label{def:embedding}
Let $G=(V, E)$ be a simple graph.
A \emph{definable embedding}, $\phi_G$,  of $G$ in $\R^2$ consists of 
\begin{enumerate}
\item
A finite subset $V_{\phi_G} \subset \R^2$, each of whose elements is labeled by a unique element of $V$
(abusing notation we will denote each element of $V_{\phi_G}$ by its label);
\item
For each edge $e = \{v,v'\} \in E$, a continuous definable embedding $\phi_e:[0,1] \rightarrow \R^2$, satisfying
$\{\phi_e(0),\phi_e(1)\}=\{v,v'\}$ and $\phi_e(t)\not\in V_{\phi_G}$, for every $t\in (0,1)$.
We denote by $\eta_e$ the image of $\phi_e$ in $\R^2$, and by $\mathring{\eta_e}$
the image of $\phi_e|_{(0,1)}$.
We denote by  $E_{\phi_G}$ the set $\cup_{e \in E} \{ \eta_e \}$. 
\end{enumerate}

For a definable embedding $\phi_G$, we will denote by $\Im(\phi_G)$ the closed and bounded 
definable set $(\displaystyle\bigcup_{e\in E}\eta_e)\cup V_{\phi_G}$. We denote by $F_{\phi_G}$ 
the set of definably connected components of 
$\R^2\setminus \Im(\phi_G)$, and refer to an element of $F_{\phi_G}$ as a \emph{face} of the embedding.
\end{definition}

\begin{definition}
\label{def:crossing}
Given a definable embedding $\phi_G$ of $G$ in $\R^2$, we define the set 
\[
EC(\phi_G) = \{(e,e') \in E(G)^2\mid e\neq e'~\text{and}~\mathring{\eta}_e\cap\mathring{\eta}_{e'}\neq\emptyset\}
\]
We define the \emph{crossing number}, $\mathrm{cr}(G)$, of $G$ by
\[
\mathrm{cr}(G) = \min_{\phi_G} \card(EC(\phi_G)),
\]
where the $\min$ is taken over all definable embeddings of $G$. 
Clearly,  $\card(EC(\phi_G))\le (\card(E(G)))^2$, for any embedding $\phi_G$ of $G$, 
and hence $\mathrm{cr}(G)$ is finite.

If $\mathrm{cr}(G)=0$ we call $G$ \emph{definably planar}.
\end{definition}

We are now in a position to prove the o-minimal version of the crossing number inequality.
We begin with some basic results.

\begin{lemma}\label{0edge}
Let $G$ be a simple connected graph. Assume that $G$ is definably planar and let 
$\phi_G$ be a definable embedding of $G$ in $\R^2$
such that $\card(EC(\phi_G))=0$.
Let $\eta\in E_{\phi_G}$ and $f\in F_{\phi_G}$.
If $\mathring{\eta}\cap \clos(f)\neq \emptyset$, then
$\eta\subset \clos(f)$.
\end{lemma}

\begin{proof}
For contradiction, assume that $\mathring \eta\cap \clos(f)\neq \emptyset$ but
$\eta\not\subset\clos(f)$. Clearly, in this case also $\mathring \eta\not\subset\clos(f)$.

By definition, $\eta$ is the image of a definable continuous 
function $\phi:[0,1]\to\R^2$.
Then for some
$a\in (0,1)$ we have $\phi(a)\in \mathring \eta\cap\clos(f)$ but
$$
\phi([a-\eps,a+\eps])\not\subset \clos(f),
$$
for every $\eps>0$ arbitrarily small.
Indeed, by our assumption there exists $t_0\in (0,1)$ such that 
$\phi(t_0)\in \clos(f)$.

Then one of the sets 
$\{t\in(t_0,1)\mid \phi(t)\not\in \clos(f)\}$ or 
$\{t\in(0,t_0)\mid \phi(t)\not\in \clos(f)\}$
is nonempty (otherwise, $\mathring\eta\subset \clos(f)$).
Suppose without loss of generality that 
$\{t\in(t_0,1)\mid \phi(t)\not\in \clos(f)\} \neq \emptyset$.
Let 
$$
a:=\inf\{t\in(t_0,1)\mid \phi(t)\not\in \clos(f)\}.
$$

Put $x_0:=\phi(a)$, and consider a cylindrical definable cell decomposition, $\mathcal{D}$,  of $\R^2$,
satisfying the frontier condition and adapted to 
$x_0$ and 
$\Im(\phi_G)$ (see \cite[Theorem 3.20]{BGV} for the existence of such a cylindrical definable decomposition).

Consider the cells of $\mathcal{D}$ that contain $x_0$ in their closure. 
By the structure of a two-dimensional cylindrical decomposition, we can order these cells (say in a counter-clockwise direction).
Let $\gamma_0,\ldots,\gamma_N=\gamma_0$ be the ordered sequence of one-dimensional cells (note that each $\gamma_i$ is either the
graph of a definable continuous function of the first coordinate or a vertical segment). Let $s_0,\ldots,s_{N-1}$ 
be the sector (i.e. two-dimensional) cells, such that $\gamma_i,\gamma_{i+1} \subset \clos(s_i)$. 
Since $\mathcal{D}$ is assumed to be adapted to $x_0$ and $\Im(\phi_G)$, there exist $i,j, 0 \leq i,j \leq N, i \neq j$,
such that  for all  small enough $\eps>0$,
$\gamma_i \supset \phi((a-\eps,a))$ and
$\gamma_j \supset \phi((a,a+\eps))$.
Without loss of generality we can assume that $i=0$.

Now observe that  $\gamma_k \cap \Im(\phi_G) = \emptyset$ for all $k \neq 0,j$. Moreover, if $\gamma_k \cap \Im(\phi_G) =\emptyset$,
then the sector cells $s_{k-1},s_k$ are contained in the same face of $\phi_G$. 
That is, $s_0,\ldots, s_{j-1}$ are contained in some $f_0\in F_{\phi_G}$, and 
$s_j,\ldots, s_{N-1}$ are contained in some $f_j\in F_{\phi_G}$ (possibly, $f_0=f_j$).
Finally, one of the sector cells must be contained in $f$ since $x_0 \in \clos(f)$. 
Thus $f\in\{f_0,f_j\}$. Suppose without loss  of generality that
$f=f_0$. 
So all the sector cells $s_0,s_1,\ldots,s_{j-1}$ must be contained in $f$, 
and this implies that $\gamma_0,\gamma_j \subset \clos(f)$, 
which is a contradiction.
\end{proof}

\begin{lemma}\label{atmost2faces}
Let $G$ be a simple connected graph. Assume that $G$ is definably planar and let 
$\phi_G$ be a definable embedding of $G$ in $\R^2$
such that $\card(EC(\phi_G))=0$.
Then every $\eta\in E_{\phi_G}$ is contained in the closure of at most 
two distinct faces $f,f'\in F_{\phi_G}$.
\end{lemma}
\begin{proof}
Let $\eta\in E_{\phi_G}$ and 
let 
$x_0 = \phi_e(a) \in\mathring\eta$.

Let $\gamma_0,\ldots,\gamma_{N-1}$ and $s_0,\ldots,s_{N-1}$ be as in the proof of Lemma
\ref{0edge}. 
Also without loss of generality assume that
$\gamma_0 \supset \phi_e((a-\eps,a))$, and let $1 < j < N$ such that
$\gamma_j \supset \phi_e((a,a+\eps))$, for all  small enough $\eps>0$.

Then, following the same argument as in the proof of Lemma \ref{0edge}, we get that
$s_0,\ldots,s_{j-1}$ must be contained in the same face, say $f$, of $\phi_G$, and 
$s_j
,\ldots,
s_{N-1}
$ must also be contained in the same face, say $f'$ of $\phi_G$ (with possibly $f=f'$). 
Moreover, it follows that $x_0$ is an interior point of $\clos(f\cup f')$, and there exists a 
definable 
neighborhood $B$ of $x_0$ contained in $\clos(f\cup f')$.
Now consider any $f'' \in F_{\phi_G}\setminus\{f,f'\}$. We have $f''\cap (f\cup f')=\emptyset$.  We claim that $x_0\not\in \clos(f'')$.
Otherwise, $f''$ has a non-empty intersection with every definable neighborhood of $x_0$, and in particular $B \cap f'' \neq \emptyset$.
But $B\cap f''$ is definably open and contained in $\clos(f\cup f')$, and hence $B \cap f'' \neq \emptyset$ implies that 
$B\cap f'' \cap (f \cup f') \neq \emptyset$ as well. However, this is a contradiction since $f''$ is disjoint from $f\cup f'$.
It follows that $x_0\not\in \clos(f'')$.
We conclude that $x_0$ is contained in the closure of at most two faces of $\phi_G$ and thus the same holds for $\eta$.
\end{proof}

\begin{lemma}\label{2faces}
Let $G$ be a simple connected graph. Assume that $G$ is definably planar and let 
$\phi_G$ be a definable embedding of $G$ in $\R^2$
such that $\card(EC(\phi_G))=0$.
Assume further that $\card(F_{\phi_G})\ge 2$. Then 
there exists $\eta\in E_{\phi_G}$ 
such that
$\eta\subset \clos(f)\cap \clos(f')$
for some distinct $f,f'\in F_{\phi_G}$.
\end{lemma}
\begin{proof}
Since $\card(F_{\phi_G})\ge 2$, there exist 
$p,q\in \R^2$ that lie in some distinct faces $f_p,f_q\in F_{\phi_G}$.
Since $V_{\phi_G}$ is a finite set of points, and hence $\HH^1(V_{\phi_G},\Q) = 0$,
we have using Proposition \ref{prop:alexander} that
$\R^2 \setminus V_{\phi_G}$ is definably connected.
Consider a 
definable path $\tau:[0,1]\to \R^2\setminus V_{\phi_G}$ 
such that $\tau(0)=p$ and $\tau(1)=q$.
Note that, for every $f\in F_{\phi_G}$, we have $\clos(f)\setminus f\subset \Im(\phi_G)$.
Indeed, for $x\in \clos(f)\setminus f$ 
it follows immediately from the definable curve selection lemma  \cite[Theorem 2]{Michel2} that
$f \cup \{x\}$ is definably connected, and since $f$ is a definably connected component of 
$\R^2 \setminus \Im(\phi_G)$, $x \in \Im(\phi_G)$.
 
In particular, for $p,q$ as above, we have $q\not\in\clos(f_p)$.

Let 
$$
t_0:=\sup\{t\in [0,1]\mid \tau(t)\in\clos(f_p)\}.
$$
Since $q\not\in \clos(f_p)$, we have $t_0<1$.
So $\tau(t_0)\in\clos(f_p)$ and $\tau(t)\not\in{\rm clos }(f_p)$, for every $t\in (t_0,1]$.

By construction, no neighborhood of $\tau(t_0)$ is contained in $f_p$, and hence 
necessarily $\tau(t_0)\in \clos(f_p)\setminus f_p$. By our argument above, 
$\tau(t_0)\in \Im(\phi_G)$.
Moreover, since the image of $\tau$ avoids vertices in $V_{\phi_G}$, we have
$\tau(t_0)\in \mathring{\eta}$, for some $\eta\in E_{\phi_G}$.
By Lemma~\ref{0edge}, we get $\eta\subset \clos(f_p)$.

Note that every definable open neighborhood $B$ of $\tau(t_0)$, as it is not contained in $f_p$, must
have a non-empty intersection with $\R^2\setminus \clos(f_p)$,  
and this intersection is a definable open set.
Since $\Im(\phi_G)$
is a one-dimensional definable set, every such neighborhood $B$
necessarily intersects $\R^2\setminus (\clos(f_p)\cup \Im(\phi_G))$.
That is, 
$$
\tau(t_0)\in \clos\left(\R^2\setminus \left(\clos(f_p)\cup \Im(\phi_G)\right)\right).
$$

Since the closure of a definable set is the union of the closures of  its definably connected components,
it follows that 
$\tau(t_0)$ is in the closure of some definably
connected component of $\R^2\setminus \left(\clos(f_p)\cup \Im(\phi_G)\right)$.

Note that $\R^2\setminus\left(\clos(f)\cup \Im(\phi_G)\right) = \cup_{f \in F_{\phi_G} \setminus \{f_p\}} f$.
That is, $\tau(t_0)\in \clos(f')$, for some $f'\in F_{\phi_G}\setminus\{f_p\}$.
By Lemma~\ref{0edge}, we have $\eta\subset \clos(f')$.
This completes the proof.
\end{proof}

\begin{lemma}\label{3edges}
Let $G$ be a simple connected graph. Assume that $G$ is definably planar and let 
$\phi_G$ be a definable embedding of $G$ in $\R^2$
such that $\card(EC(\phi_G))=0$.
If $\card(E(G))\ge 3$, 
then, for every face $f\in F_{\phi_G}$, the closure $\clos(f)$
contains at least three edges of $G$.
\end{lemma}
\begin{proof}
Fix $f\in F_{\phi_G}$, and put $Z:=\Im(\phi_G)\cap \clos(f)$.
By Lemma~\ref{0edge}, $Z$ is a union of some elements of $E_{\phi_G}\cup V_{\phi_G}$.
Let $G_f$ denote the abstract graph that corresponds to the elements
composing $Z$, in the above sense, and let $\phi_{G_f}$ denote the 
definable embedding induced by $\phi_G$ restricted to the elements of $G_f$.

Observe that necessarily $f\in F_{\phi_{G_f}}$. Indeed, a definable path from a point of $f$
to a point of $\R^2\setminus (\Im(\phi_G)\cup f)$ necessarily intersect $\clos(f)\setminus f$.

Assume first that ${\rm card}(F_{\phi_{G_f}})=1$, that is, $F_{\phi_{G_f}}=\{f\}$.
Then $\clos(f)=\R^2$, and hence $f$ intersects every open subset
of $\R^2\setminus \Im(\phi_{G_f})$. Since the intersection is open (and nonempty)
and $\Im(\phi_{G})$ is one-dimensional, $f$ must have a non-empty intersection
also with $\R^2\setminus \Im(\phi_{G})$. Hence, ${\rm card}(F_{\phi_{G}})=1$
and $\clos(f)=\R^2$.
In particular, $\eta\subset \clos(f)$, for every $\eta\in E_{\phi_G}$.
This completes the proof for this case.

Assume next that ${\rm card}(F_{\phi_{G_f}})\ge 2$.
It follows from Lemma \ref{lem:alexander} that $G_f$ contains a cycle.
Hence, $\card(E_{G_f})\ge 3$.
\end{proof}

\begin{lemma}
\label{lem:alexander}
If $T$ is a non-empty, connected graph without cycles, then $\card(F_{\phi_T})=1$, for any definable embedding $\phi_T$.
\end{lemma}

\begin{proof}
It is an immediate consequence of Definition \ref{def:embedding}, that $\Im(\phi_T)$ is a definable, closed, bounded
and definably contractible subset of $\R^2$. 
Hence, $\HH^1(\Im(\phi_T),\Q) = 0$, and the lemma follows immediately from Proposition \ref{prop:alexander}.
\end{proof}

We prove an analogue of Euler's formula for definably planar graphs.

\begin{lemma}[{\bf Euler's formula}]\label{lem:euler}
Let $G$ be a simple connected graph. Assume that $G$ is definably planar and let 
$\phi_G$ be a definable embedding of $G$ in $\R^2$
such that $\card(EC(\phi_G))=0$. Then
\begin{equation}\label{euler}
\card(V) - \card(E) + \card(F_{\phi_G}) =  2.
\end{equation}
\end{lemma}
\begin{proof}
We prove by induction on the number of faces $\card(F_{\phi_G})$.
Let $G$ be as in the statement and let $\phi_G$ be a definable embedding of $G$, such that
$\card(EC(\phi_G))=0$ and $\card(F_{\phi_G})=1$.

We claim that $G$ has no cycles.
Indeed, suppose, for contradiction, that $G$ has a cycle $(w_1,\ldots,w_r)$, and put $e_i:=\{w_i,w_{i+1}\}$, for $1\le i\le r-1$.

Let $\displaystyle{M = \cup_{i=1}^{r-1} \Im(\phi_{e_i})}$. 
Notice that each $\Im(\phi_{e_i})$ is definably homeomorphic to $[0,1]$,
and that $\Im(\phi_{e_i}) \cap \Im(\phi_{e_j})$ is empty  if $i-j \neq 1,-1 \mod r$, and is a point otherwise.
This implies using a standard Mayer-Vietoris argument that, 
\begin{eqnarray*}
\HH^i(M,\Q) &\cong& \Q, i=0,1, \\
&\cong& 0, \mbox{ else.}
\end{eqnarray*}
Proposition \ref{prop:alexander} now implies that $\R^2\setminus M$ has exactly two definably connected components.
 
This further implies that $\R^2\setminus\Im(\phi_G)$
has at least two definably connected components, contradicting our assumption.
Thus $G$ is cycle free, as claimed.

Since $G$ is connected, it is necessarily a tree.
Hence
$\card (E(G))=\card(V(G))-1$, and the identity \eqref{euler}
holds for this case. 
This proves the base case.

Assume that the lemma holds for every simple connected graph $H$ and a definable embedding 
$\phi_H$, with $\card(EC(\phi_H))=0$ and $\card(F_{\phi_H})=n$, $n\ge 1$.
Let $G$ be a simple connected graph and let $\phi_G$ be a definable embedding of $G$, such that
$\card(EC(\phi_G))=0$ and $\card(F_{\phi_G})=n+1\ge 2$. 

By Lemma~\ref{2faces},
there exist $\eta=\eta_e\in E_{\phi_G}$ and some distinct $f,f'\in F_{\phi_G}$, such that
$\eta\subset \clos(f)\cap \clos(f')$.
Consider the graph $G'$, such that $V(G')=V(G)$ and $E(G')=E(G)\setminus\{e\}$. 
Let $\phi_{G'}$ be the definable embedding that identifies with $\phi_G$ 
for all vertices and all edges excluding $e$.
As usual, let $F_{\phi_{G'}}$ denote the set of faces induced by $\phi_{G'}$.

We claim that
\begin{equation}\label{step}
F_{\phi_{G'}}=(F_{\phi_{G}}\setminus\{f,f'\})\cup \{f\cup f'\cup\mathring\eta\}
\end{equation}
Note that $U\cup \{x\}$ is definably connected, for every definably connected $U$ and $x\in \clos(U)$ 
(see the proof of Lemma~\ref{2faces}).
This implies that $f\cup f'\cup\mathring\eta\subset \R^2\setminus\Im(\phi_{G'})$ is definably connected.

Recall that a definably connected set is also definably pathwise connected. 
Consider $p,q\in \R^2$ such that $p\in f_p$, $q\in f_q$, and $f_p\neq f_q\in F_{\phi_G}$.
Assume that $f_p\cup f_q$ is contained in a definably connected component of $\R^2\setminus \Im(\phi_{G'})$.

Let $\tau:[0,1]\to \R^2\setminus \Im(\phi_{G'})$ be a definable path such that $\tau(0)=p$
and $\tau(1)=q$.
We may assume, without loss of generality, that the image of $\tau$ does not intersect any other face of $F_{\phi_G}$;
otherwise, replace $q$ with a point of this face, and restrict $\tau$ to a subinterval of $[0,1]$.
Then 
$$
\Im(\tau)\subset f_p\cup f_q\cup\mathring \eta.
$$
Moreover, $\Im(\tau)$ necessarily intersects $\mathring\eta$ in a point $x$ such that $x\in \clos(f_p)\cap 
\clos(f_q)$. By Lemma~\ref{0edge}, $\eta\subset \clos(f_p)\cap 
\clos(f_q)$. Hence, applying Lemma~\ref{atmost2faces}, we necessarily have $\{f_p,f_q\}=\{f,f'\}$.
This proves \eqref{step}.

By the induction hypothesis, we have 
$$
\card(V(G')) - \card(E(G')) + \card(F_{\phi_{G'}}) =  2,
$$
or
$$
\card(V(G)) - (\card(E(G))-1) + (\card(F_{\phi_{G}})-1) =  2.
$$
This completes the proof.
\end{proof}

It follows from Lemma~\ref{3edges} and Lemma~\ref{lem:euler} that, for 
every connected definably planar graph $G$, with $\card(V(G)) \geq 4$,
we have 
\begin{equation}\label{eulercor}
\card(E) \le 3\cdot\card(V) - 6.
\end{equation}
Indeed, let $\phi_G$ be a definable embedding of $G$ 
in $\R^2$ such that $\card(EC(\phi_G))=0$.
By Lemma~\ref{3edges}, the closure of every $f \in F_{\phi_G}$ contains at least three distinct edges in $E_{\phi_G}$, and,
for every $e\in E$, $\eta_e$ is contained in the closures of at most two elements of $F_{\phi_G}$, by Lemma~\ref{atmost2faces}. 
Thus, $3\cdot\card(F_{\phi_G})\le 2\cdot\card(E)$.
Combined with Lemma~\ref{lem:euler}, 
the inequality \eqref{eulercor} follows.

We conclude with  an extension of the crossing number inequality, applicable to definable embeddings of graphs in $\R^2$.
\begin{lemma}[{\bf crossing number inequality}]\label{lem:cross}
Let $G$ be a simple connected graph, such that $\card(E(G))> 4\cdot \card(V(G))$. 
Then 
\begin{equation}\label{cross}
{\rm cr}(G)\ge \frac{\card(E)^3}{64\cdot\card(V)^2}.
\end{equation}
\end{lemma}

The proof of Lemma \ref{lem:cross} is now very standard (see for example, \cite[p. 285]{Alon-Spencer})  
using the basic results on embeddings of graphs over general o-minimal
structures developed above. We include it in the appendix (Section \ref{sec:appendix}).

\begin{proof}[Proof of Lemma~\ref{lem:lines}]
Let $\Pi,\Gamma$ be as in the statement.

There exists a constant $C=C(V)$, such that each $\gamma \in \Gamma$ can be partitioned into a disjoint 
union of at most $C$ 
definable curves, 
each definably homeomorpic to the open interval $(0,1)$,
and at most $C$ points. 
Let $\Gamma_\gamma$ denote this set of curves. Let 
$\Gamma' =\cup_{\gamma \in \Gamma} \Gamma_\gamma$. By construction,
$$
I(\Pi,\Gamma)\le I(\Pi,\Gamma') +Cn
$$ and 
$\card(\Gamma')\le Cn$.

Let $\Pi'' \subset \Pi$ be the subset defined by
\[
\Pi'' = \{p \in \Pi \mid \card(\{\gamma \in \Gamma'\mid p \in \gamma\}) \leq 1\}.
\]

Let $\Pi' = \Pi \setminus \Pi''$.

Note that each point in $\Pi''$ contributes at most one incidence to our counting, and thus 
\begin{equation}\label{1}
I(\Pi,\Gamma)\le I(\Pi',\Gamma')+m+Cn.
\end{equation}

We construct a graph $G=(V,E)$ as follows.
Every point of $\Pi'$ corresponds to a vertex of $G$. 
Every pair of points $p,q\in\Pi'$ that lie \emph{consecutively} on a curve $\gamma\in\Gamma'$
are connected by an edge in $G$.
Note that a pair of vertices can lie consecutively on more than one curve of $\Gamma'$.
Nevertheless, such a pair will contribute only one edge to $G$.
Using our assumption 
\eqref{itemlabel:lines:a}
we get
\begin{equation}\label{2}
I(\Pi',\Gamma')\le k\cdot\card(E) +n,
\end{equation}
where the additive factor $n$
compensates for the at most one incidence that we lose on each curve 
(a curve incident to $r$ points, contributes exactly $r-1$ edges to the graph, counting with multiplicity).

Let $G_i=(V_i,E_i)$ denote the 
connected components of $G$ (i.e. maximal connected induced subgraphs).
Recall that each vertex of $V$ corresponds (injectively) to a point of $\Pi'$.
Let $\Gamma'_i$ denote the subset of curves $\gamma\in \Gamma'$, such that 
$\gamma$ is incident to one of the points that corresponds to a vertex of $V_i$.
Put $m_i:=\card (V_i)$ and $N_i:=\card(\Gamma'_i)$.
Observe that $V=\bigcup_iV_i$, $E=\bigcup_iE_i$, and
$\Gamma'=\bigcup_i\Gamma'_i$ are disjoint unions, and
hence
\begin{equation}\label{sum}
\sum_im_i=m~~, ~~\card(E)=\sum_i\card(E_i),~~\text{ and }~~\sum_iN_i=\card (\Gamma')\le Cn.
\end{equation}

Let $i$ be fixed.
Note that property \eqref{itemlabel:lines:b} in Lemma~\ref{lem:lines} implies that 
${\rm cr}(G_i)\le \binom{N_i}{2}k$, since each crossing is induced by a pair of
curves of $\Gamma'_i$ that intersect. 
By definition, the graph $G_i$ is simple and connected.
Applying Lemma~\ref{lem:cross} to $G_i$, we get
\begin{equation}\label{4}
\card(E_i)\le C_0 k^{1/3}m_i^{2/3}N_i^{2/3}+4m_i,
\end{equation}
for some absolute constant $C_0$.
Combining \eqref{sum}, \eqref{4}, and H\"older's inequality, we get
\begin{align*}
\card(E)
&=\sum_i\card(E_i)\\
&\le C_0 k^{1/3}\sum_im_i^{2/3}N_i^{2/3}+4\sum_i m_i\\
&\le C_0 k^{1/3}(Cn)^{1/3}\sum_im_i^{2/3}N_i^{1/3}+4m\\
&\le C_0 k^{1/3}(Cn)^{1/3}m^{2/3}\left(\sum_i N_i\right)^{1/3}+4m\\
&\le C'k^{1/3}m^{2/3}n^{2/3}+4m.\\
\end{align*}

Finally, the inequalities \eqref{1} and \eqref{2} imply 
$$
I(\Pi,\Gamma)\le C_{V,k}\left(m^{2/3}n^{2/3}+m+n\right),
$$
where $C_{V,k}$ is a constant that depends only on  $V$ and the parameter $k$.\end{proof}

\section{Proof of Theorem~\ref{thm:main}}
\label{sec:proof-of-main}

We will need the following lemma.

\begin{lemma}\label{lem:prod}
Let $S_1\subset\R^{d_1}$, $S_2\subset \R^{d_2}$ be definable sets of dimensions $k_1,k_2$, respectively.
Let $W\subset S_1\times S_2$ be a $(k_1+k_2)$-dimensional definable subset of $S_1\times S_2$.
Then there exist $S_i' \subset S_i$ definable with $\dim S_i'=k_i$, for $i=1,2$, such that
$S_1'\times S_2' \subset W$.
\end{lemma}
\begin{proof}
For each $i=1,2$, consider a cylindrical definable cell decomposition of $\R^{d_i}$ adapted to $S_i$.
It is easy to see that $\dim (W\cap (C\times D))=k_1+k_2$,
for some cells $C\subset S_1$ and $D\subset S_2$ of the respective decompositions,
where $\dim (C)=k_1$ and $\dim (C)=k_2$.
Indeed, otherwise we would have had $\dim(W\cap (S_1\times S_2))<k_1+k_2$,
contradicting our assumption on $W$.

By properties of the decomposition, there exist definable homeomorphisms $\theta_C:C\to \R^{k_1}$
and $\theta_D:D\to\R^{k_2}$. Let 
 $f = \theta_C \times \theta_D :C\times D\to \R^{k_1}\times \R^{k_2}$ be the homeomorphism induced by $\theta_C$ and 
 $\theta_D$.
Then $W':=f(W\cap (C\times D))$, is a $(k_1+k_2)$-dimensional definable set contained in $\R^{k_1}\times \R^{k_2}$.
In particular, $W'$ contains an open set of $\R^{k_1}\times \R^{k_2}$, and hence there exist open sets
$I\subset \R^{k_1}$ and $J\subset \R^{k_2}$, such that $I\times J\subset W'$.
Letting $S_1':=\theta_C^{-1}(I)$ and $S_2':=\theta_D^{-1}(J)$ the lemma follows.
\end{proof}

Let us introduce some notation.
For $p,q\in E$, we define
$$
\gamma_q:=\pi_1(\pi_2^{-1}(q)\cap V)
$$
and 
$$
\gamma^*_p:=\pi_2(\pi_1^{-1}(p)\cap V).
$$
For $i=0,1,2$, we let  
$$Y_i:=\{q\mid \dim(\gamma_q)= i\}
$$
and 
$$X_i:=\{p\mid \dim(\gamma^*_p)=i\};
$$
note that $\dim(\gamma_q),\dim(\gamma^*_p)\le 2$, for every $p,q$.

\begin{lemma}\label{lem:Usmall}
Let $V$, $X_1$, $Y_1$ be as above.
Let $P,Q\subset E$ be finite subsets.
Assume that $P\subset X_1$, $Q\subset Y_1$, and that,
for every $p\in P, q\in Q$, each of the sets
$$
U_{1,q}:=\{q'\in Y_1\mid \dim(\gamma_q\cap \gamma_{q'})\ge 1\},
$$ 
$$
U_{2,p}:= \{p'\in X_1\mid \dim(\gamma^*_p\cap \gamma^*_{p'})\ge 1\}
$$ 
is zero-dimensional.
Then property 
\eqref{itemlabel:main:i}
in the statement of Theorem~\ref{thm:main} holds.
\end{lemma}
\begin{proof}
Our assumption on 
$U_{1,q}$ and $U_{2,p}$ implies that for some constant $M$,
which depends only on $V$, 
$$
|U_{1,q}|\le M~~\text{and}~~|U_{2,p}|\le M,
$$
for every $p\in P$, $q\in Q$.

Let $G = G(P)$ denote the graph defined by 
\begin{eqnarray*}
V(G) &=& P, \\
E(G) &=& \{(p,p') \in V(G(P)) \mid \dim \left(\pi_2(\pi_1^{-1}(p) \cap V) \cap \pi_2(\pi_1^{-1}(p')\cap V) \right)=1\}.
\end{eqnarray*}

Notice that $G$ has maximum vertex degree at most $M$, 
so we can color the graph with $M+1$ colors. In other words, we can partition $P$ into $M+1$ sets $P_i$, 
so that for any pair of distinct $p,p'\in P_i$ the set
$$
\pi_2(\pi_1^{-1}(p)\cap V) \cap  \pi_2(\pi_1^{-1}(p')\cap V)
$$
is finite, and hence bounded by some constant $N$, which depends only on $V$.

Similarly, there exists a partition of $Q$ into at most $M+1$ sets $Q_j$ so that
$$
\pi_1(\pi_2^{-1}(q)\cap V)\cap \pi_1(\pi_2^{-1}(q')\cap V)
$$
is finite and of cardinality at most a constant depending only on $V$ (which we can again take to be $N$), 
for every pair of distinct $q,q'\in Q_j$.

To prove property 
\eqref{itemlabel:main:i},
it suffices to show that there exists a constant $C_V$ depending only on $V$, such that
$$
\card(V\cap (P_i\times Q_j)) \leq C_V\cdot (\card(P_i)^{2/3}\cdot \card(Q_j)^{2/3}+\card(P_i)+\card(Q_j)),
$$
for every pair $1\le i,j\le M$.

Fix $1\le i,j\le M$ and put $\Pi:=Q_j$ and 
$
\Gamma:=\{\gamma_{p}\mid p\in P_i\}
$
where
$$
\gamma_{p}:=\pi_2(\pi_1^{-1}(p)\cap V).
$$
By definition, $q \in \gamma_{p}$ if and only if $(p,q)\in V$.
Note also that 
$$
\card(\gamma\cap \gamma') \le N,
$$
for every pair of distinct $\gamma,\gamma'\in \Gamma$, 
and that
$$
\card(\{\gamma\in \Gamma\mid p,p'\in \gamma\})\le N,
$$
for every pair of distinct $p,p'\in \Pi$.

Now apply Lemma \ref{lem:lines}.
\end{proof}

\begin{proof}[Proof of Theorem~\ref{thm:main}]
First observe that if $\dim(V) = 4$, then 
it follows immediately from Lemma \ref{lem:prod}, with $S_1=S_2  =\R^2$ and 
$W = V$, that property \eqref{itemlabel:main:ii} of the theorem holds. So we can assume that $\dim(V) \leq 3$.

Our strategy is to try to bound the quantity $\card(V\cap(P\times Q))$, as in property \eqref{itemlabel:main:i},
and, in each step, prove that if the bound fails, then property \eqref{itemlabel:main:ii} of the theorem must hold. 
Since 
$\dim(V) \leq 3$, 
we have $\dim X_i, \dim Y_i\le 3-i$, for each $i=1,2$.
We bound the cardinality of 
$V\cap(P\times (Q \cap Y_2))$ 
(unless property \eqref{itemlabel:main:ii}  holds).
Let 
$$
Y^*_2:=\{p\in E\mid \dim(\gamma_p^*\cap Y_2)\ge 1\}.
$$
Note that the definition of $Y^*_2$ is independent of $P,Q$ and depends only on $V$.

Assume first that $\dim Y^*_2\ge1$. Note that this implies in particular that $\dim Y_2=1$.
Then the set
$$
\{(p,q)\in F\mid p\in Y_2^*, q\in Y_2, (p,q)\in V\}\subset V
$$
is of dimension $\dim (Y_2^*)+1$ and it is contained in the 
the set $Y_2^*\times Y_2$. Recalling that $\dim(Y_2)=1$, and applying Lemma~\ref{lem:prod}
to conclude that there exists definable sets  $\alpha,\beta \subset \R^2$ with 
$\dim\alpha=\dim Y_2^*\ge 1$ and $\dim \beta=1$, such that $\alpha \times \beta \subset V$.
This completes the proof of the theorem for this case.

Hence, we can assume $\dim Y^*_2=0$.
Since $Y^*_2$ is finite, 
$\card(Y^*_2)$ is bounded from above by a constant $C_2'$ which depends only on $V$.
Hence, $\card (P\cap Y^*_2) \leq C_2'$,
and
in particular, 
$$
\card\left(V\cap \left((P\cap Y^*_2)\times (Q\cap Y_2)\right)\right)\le C_2'\cdot \card(Q).
$$

Note also that for $p\in P\setminus Y_2^*$ we have 
$\card(\gamma^*_p\cap Y_2) \leq C_2''$, where $C_2''$ is some constant that depends only on $V$. Hence,
$$
\card(V\cap ((P\setminus Y_2^*)\times (Q\cap Y_2)))=C_2'' \cdot \card(P).
$$
We conclude that
$$
\card(V\cap (P\times (Q\cap Y_2)))\le C_2 \cdot (\card(P)+\card(Q)),
$$
for some constant $C_2$ that depends only on $V$.

Symmetrically, it follows that 
$$
\card(V\cap ((P\cap X_2)\times Q))\le C_3 \cdot (\card(P)+\card(Q)),
$$
for some constant $C_3>0$ which depends only on $V$, 
unless property \eqref{itemlabel:main:ii}  of the theorem holds.

We are left to bound the cardinality of $V\cap \left((P\cap X_1)\times (Q\cap Y_1)\right)$.
Without loss of generality, we can assume in what follows, that $P\subset X_1$
and $Q\subset Y_1$.

Let 
\begin{eqnarray*}
U_1 &:=& \{(q,q') \in Y_1\times Y_1 \mid \dim (\gamma_q\cap\gamma_{q'}) \geq 1\}, \\
U_2 &:=& \{(p,p') \in X_1\times X_1 \mid \dim (\gamma^*_p\cap \gamma^*_{p'}) \geq 1\}.
\end{eqnarray*}

For $q \in Y_1$ (resp. $p \in X_1$) we will denote by $U_{1,q}$ (resp. $U_{2,p}$), the definable set
$\pi_1^{-1}(q) \cap U_1$ (resp. $\pi_1^{-1}(p)\cap U_2$).

Assume first that $\dim (U_{1,q})\ge1$, for some fixed $q\in Y_1$.
By definition, $\dim(\gamma_q\cap \gamma_{q'})\ge 1$ for all $q'\in U_{1,q}$. Thus, the set 
$$
\{(p,q')\in \gamma_q\times U_{1,q}\mid (p,q')\in V \}
$$
is of dimension $\dim(U_{1,q})+1$.
By Lemma~\ref{lem:prod}, this implies that there exist definable  $\alpha,\beta\subset E$,
such that $\dim \alpha=\dim U_{1,q}\ge1$, $\dim\beta=1$ and $\alpha\times \beta\subset V$.
Thus property \eqref{itemlabel:main:ii}  of the theorem holds for this case.
Symmetrically, property \eqref{itemlabel:main:ii} of the theorem holds in case $\dim (U_{2,p})\ge1$ for some $p\in X_1$.

Therefore, we can assume that $\dim (U_{2,p})=\dim (U_{1,q})=0$, for every $p\in X_1$
and $q\in Y_1$.
Property \eqref{itemlabel:main:i} of the theorem now follows from Lemma~\ref{lem:Usmall}.
\end{proof}

\bibliographystyle{amsplain}
\bibliography{master}

\section{Appendix}
\label{sec:appendix}
\begin{proof}[Proof of Lemma \ref{lem:cross}] 
Let $G$ be as in the statement.
Let $\phi_G$ be a definable embedding of $G$, with $\card(EC(\phi_G)) = \mathrm{cr}(G)$.

Let $E''$ be a minimal subset of $E$ with the property that $E''$ has a nonempty
intersection with each element of $EC(\phi_G)$. Note that it follows from this definition that
$\card(E'') \leq \card(EC(\phi_G)) = {\rm cr}(G)$.
Let $E' = E \setminus E''$.

By construction, the graph $G' = (V,E')$ is definably planar and
has the same number of vertices as $G$. Moreover, $\card(E') \ge \card(E)-{\rm cr}(G)$.
By the inequality \eqref{eulercor} applied to the graph $G'$, we also have
$\card(E') \le 3\cdot \card(V) -6$. Thus, 
\begin{equation}\label{weakcross}
{\rm cr}(G)\ge \card(E) -3 \cdot \card(V).
\end{equation}
So \eqref{weakcross} holds for any simple 
graph $G=(V,E)$ with $\card(V) \ge 4$.

We now use a probabilistic argument to obtain a sharper inequality.
Let $G$ be as above and consider $G''$ a random subgraph of $G$ obtained 
by taking each vertex of $G$ to lie in $G''$ independently with probability $p$, and 
taking an edge of $G$ to lie in $G''$ if and only if its two vertices were chosen to lie in $G''$.
Let $v''$ and $e''$ denote the number of edges and vertices of $G''$, respectively. 
By \eqref{weakcross}, we have ${\rm cr}(G'')\ge e''-3\cdot v''$.
Taking expectation, we have
$$
\mathbb E[{\rm cr}(G'')]\ge \mathbb E[e'']-3\cdot\mathbb E [v''],
$$
which implies
$$
p^4\cdot{\rm cr}(G)\ge p^2\cdot \card(E)-3 p\cdot\card(V).
$$
Now if we assume $\card(E)>4\cdot\card(V)$ and set $p = 4\cdot \card(V)/\card(E)$, we obtain
\begin{equation}
{\rm cr}(G)\ge\frac{\card(E)^3}{64\cdot\card(V)^2}.\qedhere
\end{equation}
\end{proof}

\end{document}